\documentclass[12pt]{amsart}
\usepackage{fullpage}
\usepackage{colonequals}
\usepackage{tikz}
\usetikzlibrary{decorations.pathreplacing}
\usepackage{amsmath,amssymb,amsthm,verbatim,mathtools,mathrsfs}
\usepackage{xspace}
\usepackage[unicode]{hyperref}
\usepackage{expl3} 
\usepackage{xparse,l3keys2e}
\usepackage{wordle}
\usepackage{tabularx,vcell}
\definecolor{darkred}{rgb}{0.7,0,0}
\definecolor{highlightcircle}{rgb}{1,1,0}
\usepackage{url}

\definecolor{multizero}{rgb}{1,1,1}
\definecolor{multione}{rgb}{0,0,0}
\newcommand{\multizeroname}{white\xspace}
\newcommand{\multionename}{black\xspace}

\newcommand{\notkappa}{c}
\newcommand{\notlambda}{b}

\newtheorem{theorem}{Theorem}[section]
\newtheorem{proposition}[theorem]{Proposition}

\newtheorem{lemma}[theorem]{Lemma}
\newtheorem{corollary}[theorem]{Corollary}
\newtheorem{conjecture}[theorem]{Conjecture}
\theoremstyle{definition}
\newtheorem{definition1}[theorem]{Definition}
\newtheorem{remark1}[theorem]{Remark}

\DeclareTextCommand{\textover}{PU}{\9040\076}

\newcommand{\Esz}{E_{s,z}}
\newcommand{\noell}{{(s_1, \ldots, \widehat{s_\ell}, \ldots, s_n)}}
\newcommand{\Dmult}[1]{{W_{(s_1, \ldots, s_{#1})}}}

\newcommand{\mysubsection}[1]{\medskip\subsection{#1}}
\DeclareMathOperator\exc{exc}

\DeclareMathOperator\Exc{Exc}

\DeclareMathOperator\fixed{fp}

\newcolumntype{C}{>{\centering\arraybackslash}X}

\newcommand\classic{permutation Wordle\xspace}
\newcommand\rainbow{multicolor permutation Wordle\xspace}
\newcommand\basicstrat{\textnormal{\textsc{CircularShift}}\xspace}

\newcommand\floor[1]{\left\lfloor{#1}\right\rfloor}

\begin{document}

\title{Permutation wordle}


\author[Samuel Kutin]{Samuel~Kutin}
\address{Center for Communications Research, 805 Bunn Drive, Princeton,
NJ 08540-1966, USA}
\email{kutin@ccr-princeton.org}

\author[Lawren Smithline]{Lawren~Smithline}
\address{Center for Communications Research, 805 Bunn Drive, Princeton,
NJ 08540-1966, USA}
\email{lawren@ccr-princeton.org}

\begin{abstract}
  We introduce a guessing game, ``permutation Wordle'',
  in which a guesser attempts to recover
   a setter's hidden
  permutation of the set $\{1, \ldots, n\}$. In each round, the guesser
  guesses a permutation, and, as in the game Wordle, learns
  which entries are correct.  We describe a natural strategy,
  which we conjecture is optimal. We show that the number of
  permutations this strategy solves in $k+1$ rounds is the Eulerian number $A(n,k)$.

  We also describe an extension to multicolor permutations: the setter chooses a
  permutation and also a coloring of $\{1, \ldots, n\}$.  We generalize our
    strategy, give a recurrence for the number of $s$-colored permutations solved
    in $k+1$ rounds, and relate these ``Wordle numbers'' to the Eulerian numbers
    and to higher-order Eulerian numbers.
\end{abstract}

\maketitle

\section{Introduction}
\label{sec:intro}

Jon Scieska~\cite{math-curse} writes,
``You can think of almost everything as a math problem.''  We present an example of
how this happens. We start with the game Wordle, turn it into a guessing game with
permutations, and find connections to a well-studied sequence:
the Eulerian numbers.

In 2021, Josh Wardle~\cite{Wordle-wikipedia} introduced the online game ``Wordle''.
In 2022, Wordle was purchased by {\emph{The New York Times}}.  By 2025, Wordle had been played
billions of times.

The goal of Wordle is to discover a secret five-letter word within six guesses.
In each round, the guesser chooses a five-letter English word
and learns which letters are
correct (in the secret word in the same place),
which are misplaced (in the secret word in a different place),
and which are wrong (not in the secret word at
all).\footnote{Wordle has well-defined rules for handling
repeated letters in the secret
word or in the guess, but these are not relevant to this paper.}

One can find strategy tips for Wordle online.  Optimal play requires an
understanding of the set of words in the language.  Even partial information
about the language helps:
for example, typical patterns of consonants and vowels, or
common word endings.  In order to turn Wordle into a math problem, we make it more abstract.

Our new game is ``\classic''.
Choose a whole number $n$.  A ``word'' is a permutation of the set
$\{1, \ldots, n\}$: each number appears exactly once.
The secret is a permutation, and
each guess is also a
permutation. The feedback for \classic is simpler than the feedback for
original Wordle: no number is wrong.  Each number is either correct or misplaced,
because each appears somewhere in the secret.

What is the optimal strategy? How many guesses do we need on average?
More generally, what is the probability that we solve after exactly $k$ wrong guesses?
In Section~\ref{sec:circular-shift} we offer a strategy, ``\basicstrat'', which we conjecture is the
best possible.  In Section~\ref{sec:eulerian} we show that \basicstrat solves
$A(n,k)$ of the $n!$ permutations in $k+1$ guesses, where $A(n,k)$ are the
\emph{Eulerian numbers}, a sequence that
frequently arises in the analysis of permutations~\cite{petersen}.  

The appearance of Eulerian numbers is a sign that we have found a ``natural''
math problem from Wordle---it is a satisfying answer. Another sign is that this
permutation-guessing problem seems to keep arising. The authors first considered it not
because of Wordle, but as an abstraction of permutation-guessing games on Sporcle.com.
In July 2025, Richard Stanley~\cite{Stanley2025} posed the same question;
he came to it from a guessing game in the app {\emph{Royal Match}}.

We can generalize this problem. In Section~\ref{sec:rainbow} we introduce a ``multicolor'' version of the game.
We can visualize \classic as selecting an order for $n$ cards bearing the ranks
$1$ through $n$.  We generalize to a larger deck of cards,
where each rank is present in more than one color;
the secret for \rainbow  corresponds to
an ordered set of $n$ cards where each number
appears exactly once.  The feedback for the guesser is still a set of which positions are
correctly filled, but with a twist: both the number and color must be correct.  We
adapt \basicstrat, and we show in Section~\ref{sec:higher}
that this multicolor problem leads to higher-order Eulerian numbers.

There are connections between Wordle and the classic board game Mastermind; in
Section~\ref{sec:mastermind} we discuss the existing literature on Mastermind, particularly
when the secret is known to be a permutation.
Finally, we state some open problems in Section~\ref{sec:open}.

\section{\basicstrat}
\label{sec:circular-shift}

We describe a simple strategy, \basicstrat, that solves \classic
in at most $n$ guesses.\footnote{Li and Zhu~\cite{LiZhu2024} showed that $n$ guesses are necessary.
See Section~\ref{sec:mastermind} for more discussion.}  Let $[n]$ denote the set $\{1, \ldots, n\}$,
and let $S_n$ denote the group of permutations on $[n]$.
For  now, we act if the setter chooses the secret permutation
$\sigma$ uniformly at random.  We show in
  Section~\ref{sec:avg-case} that we can restrict our analysis to this case, because the guesser can employ a randomized strategy.

  We denote the guesser's guesses by $\{\gamma_r\}$ for $r \ge 1$.
  In round $r$, the feedback
  \begin{equation}
    \label{eq:def_Ur}
    U_r = \{i \mid \gamma_r(i) = \sigma(i)\}
    \end{equation}
  is the set of positions where $\gamma_r$ agrees with the setter's permutation $\sigma.$
  Equivalently, $U_r$ is the set of fixed points of the permutation
  $\sigma^{-1} \gamma_r$.
  The guesser may use the information $U_k$ for $1 \leq k \leq r$
  in making the choice of $\gamma_{r+1}$.
  
  \mysubsection{Small $n$}
\label{sec:circular-shift-small-n}
  
We begin with the case $n=3$.  All first guesses are equivalent, so we may as well choose the identity permutation
as our first guess $\gamma_1$.
In response to this first guess, we learn the set of fixed points of $\sigma$.

    \begin{figure}[h]
      \begin{center}
        \begin{tabularx}{\textwidth}{CCC}
         \savecellbox{\scalebox{0.7}{\begin{wordle}[present color=WordleAbsent]{123}
      123
    \end{wordle}}} &
    \savecellbox{\scalebox{0.7}{\begin{wordle}[present color=WordleAbsent]{132}
          123
          132
    \end{wordle}}} &
    \savecellbox{\scalebox{0.7}{\begin{wordle}[present color=WordleAbsent]{213}
          123
          213
    \end{wordle}}} \\ [-\rowheight]
    \printcelltop & \printcelltop & \printcelltop \\
     & & \\
    \savecellbox{\scalebox{0.7}{\begin{wordle}[present color=WordleAbsent]{231}
                   123
                   312
                   231
    \end{wordle}}} &
    \savecellbox{\scalebox{0.7}{\begin{wordle}[present color=WordleAbsent]{312}
          123
          312
    \end{wordle}}} &
    \savecellbox{\scalebox{0.7}{\begin{wordle}[present color=WordleAbsent]{321}
          123
          321
    \end{wordle}}} \\ [-\rowheight]
    \printcelltop & \printcelltop & \printcelltop
    \end{tabularx}
      \end{center}
      \caption{Permutation Wordle when $n=3$.  We solve $1$ permutation in one
        round, $4$ in two rounds, and $1$ in three rounds.}
      \label{fig-n3}
    \end{figure}

    Figure~\ref{fig-n3} shows the different scenarios.
    There are three possibilities, depending on the number
    of fixed points of $\sigma.$
    If $\sigma$ has three fixed points, then $\sigma$ is the identity and we are done in
    one guess. If $\sigma$
    has exactly one fixed point, then $\sigma$ is one of the three permutations that swaps two
    values, and the fixed point identifies the swap;
    in round $2$, we swap those two values and get the correct answer.
    If $\sigma$ has no fixed points, then $\sigma$ is one of the two $3$-cycles,
    but we do not know which.  Our best strategy is
    to guess one of the $3$-cycles in round $2$; we will either be correct, or we will need
    a third guess.  Overall, the probability of finding $\sigma$ in one round is $1/6$;
    in two rounds, $4/6$; in three rounds, $1/6$.

    The analysis for $n=4$ is more complicated. Again, let $\gamma_1$ be the identity.
    If $\sigma$ is the identity, we finish in one round. If $\sigma$ has one or two fixed
    points, we can just keep those values the same in $\gamma_r$ for all $r > 1$; this lets
    us reduce to the $n=2$ or $n=3$ case, and we finish in a total of two or three rounds.

    What if $\sigma$ is one of the nine {\em derangements},
    meaning $\sigma$ has no fixed points?  One can work out (by hand or by computer
    exhaustion) that
    the optimal strategy is to let $\gamma_2$ be a $4$-cycle.\footnote{This
    strategy is optimal both in the sense of maximizing the probability of
    solving within a particular number of rounds
    and in the sense of minimizing the expected number of rounds.}
      If $\sigma = \gamma_2$, we are done in two
    rounds; if $\sigma^{-1} \gamma_2$ has one or two fixed points, we can deduce $\sigma$
    in the third round; otherwise, the remaining possibilities for $\sigma$ are $\gamma_2^2$ and
    $\gamma_2^3$, so we choose one of them and solve in either three or four
    rounds.  Figure~\ref{fig-n4} depicts this overall strategy.

    \begin{figure}
      \begin{center}
         \begin{tabularx}{0.94\textwidth}{CCCCCC}
    \savecellbox{\scalebox{0.6}{\begin{wordle}[present color=WordleAbsent]{1234}
      1234
    \end{wordle}}} &
    \savecellbox{\scalebox{0.6}{\begin{wordle}[present color=WordleAbsent]{1243}
          1234
          1243
    \end{wordle}}} &
    \savecellbox{\scalebox{0.6}{\begin{wordle}[present color=WordleAbsent]{1324}
          1234
          1324
    \end{wordle}}} &
    \savecellbox{\scalebox{0.6}{\begin{wordle}[present color=WordleAbsent]{1342}
          1234
          1423
          1342
    \end{wordle}}} &
    \savecellbox{\scalebox{0.6}{\begin{wordle}[present color=WordleAbsent]{1423}
          1234
          1423
    \end{wordle}}} &
    \savecellbox{\scalebox{0.6}{\begin{wordle}[present color=WordleAbsent]{1432}
          1234
          1432
    \end{wordle}}} \\ [-\rowheight]
    \printcelltop & \printcelltop & \printcelltop & \printcelltop & \printcelltop & \printcelltop\\
    &&&&&\\
    \savecellbox{\scalebox{0.6}{\begin{wordle}[present color=WordleAbsent]{2134}
          1234
          2134
    \end{wordle}}} &
    \savecellbox{\scalebox{0.6}{\begin{wordle}[present color=WordleAbsent]{2143}
          1234
          4123
          2143
    \end{wordle}}} &
    \savecellbox{\scalebox{0.6}{\begin{wordle}[present color=WordleAbsent]{2314}
          1234
          3124
          2314
    \end{wordle}}} &
    \savecellbox{\scalebox{0.6}{\begin{wordle}[present color=WordleAbsent]{2341}
          1234
          4123
          3412
          2341
    \end{wordle}}} &
    \savecellbox{\scalebox{0.6}{\begin{wordle}[present color=WordleAbsent]{2413}
          1234
          4123
          2413
    \end{wordle}}} &
    \savecellbox{\scalebox{0.6}{\begin{wordle}[present color=WordleAbsent]{2431}
          1234
          4132
          2431
    \end{wordle}}} \\ [-\rowheight]
    \printcelltop & \printcelltop & \printcelltop & \printcelltop & \printcelltop & \printcelltop\\
    &&&&&\\
    \savecellbox{\scalebox{0.6}{\begin{wordle}[present color=WordleAbsent]{3124}
          1234
          3124
    \end{wordle}}} &
    \savecellbox{\scalebox{0.6}{\begin{wordle}[present color=WordleAbsent]{3142}
          1234
          4123
          3142
    \end{wordle}}} &
    \savecellbox{\scalebox{0.6}{\begin{wordle}[present color=WordleAbsent]{3214}
          1234
          3214
    \end{wordle}}} &
    \savecellbox{\scalebox{0.6}{\begin{wordle}[present color=WordleAbsent]{3241}
          1234
          4213
          3241
    \end{wordle}}} &
    \savecellbox{\scalebox{0.6}{\begin{wordle}[present color=WordleAbsent]{3412}
          1234
          4123
          3412
    \end{wordle}}} &
    \savecellbox{\scalebox{0.6}{\begin{wordle}[present color=WordleAbsent]{3421}
          1234
          4123
          3421
    \end{wordle}}} \\ [-\rowheight]
    \printcelltop & \printcelltop & \printcelltop & \printcelltop & \printcelltop & \printcelltop\\
    &&&&&\\
    \savecellbox{\scalebox{0.6}{\begin{wordle}[present color=WordleAbsent]{4123}
          1234
          4123
    \end{wordle}}} &
    \savecellbox{\scalebox{0.6}{\begin{wordle}[present color=WordleAbsent]{4132}
          1234
          4132
    \end{wordle}}} &
    \savecellbox{\scalebox{0.6}{\begin{wordle}[present color=WordleAbsent]{4213}
          1234
          4213
    \end{wordle}}} &
    \savecellbox{\scalebox{0.6}{\begin{wordle}[present color=WordleAbsent]{4231}
          1234
          4231
    \end{wordle}}} &
    \savecellbox{\scalebox{0.6}{\begin{wordle}[present color=WordleAbsent]{4312}
          1234
          4123
          4312
    \end{wordle}}} &
    \savecellbox{\scalebox{0.6}{\begin{wordle}[present color=WordleAbsent]{4321}
          1234
          4123
          4321
    \end{wordle}}} \\ [-\rowheight]
    \printcelltop & \printcelltop & \printcelltop & \printcelltop & \printcelltop & \printcelltop
  \end{tabularx}   
      \end{center}
      \caption{Permutation Wordle when $n=4$.  We solve $1$ permutation in one
        round, $11$ in two rounds, $11$ in three rounds, and $1$ in four rounds.}
      \label{fig-n4}
      \end{figure}

    \mysubsection{An adversarial setter}
\label{sec:avg-case}
    
    Before we go on to general $n$, we consider the possibility of an
    adversarial setter.  We have described a deterministic strategy, which makes sense
    if the setter is choosing $\sigma$ uniformly at random.
    What if the setter uses knowledge of the guesser's (possibly randomized) strategy to
    choose the secret permutation?
    Can the setter make the game harder by choosing
    $\sigma$ from some other distribution?\footnote{
    We do not consider the case (common in analyses of Mastermind) of the adaptive
    setter who changes
    the secret permutation during the game to make the guesser take as long as possible.}

    For \classic, the answer turns out to be ``no'': we may as well consider a uniform
    random $\sigma$.  We can always randomize our strategy by introducing an
    extra, uniformly chosen permutation $\tau$.  Instead of the sequence of guesses
    $\gamma_1$, $\gamma_2$, $\ldots$, we guess $\gamma'_r = \tau\gamma_r$.
    The feedback in each round is the set of fixed points of $\sigma^{-1} {\gamma'_r}
    = (\sigma^{-1} \tau) \gamma_r$; this is equivalent to applying
    a deterministic strategy to $\tau^{-1}\sigma$, which is a uniform random permutation
    no matter the setter's strategy.

    For the rest of this paper, we consider deterministic strategies, but the reader
    can imagine that in practice we would multiply by a uniform random $\tau$.  The
    key is that we can analyze average-case behavior.  In contrast, in the classic
    game Mastermind, not all of the setter's codes are equivalent, so the literature
    tends to focus on worst-case behavior.  See Section~\ref{sec:mastermind} for more
    discussion.

    \mysubsection{The general algorithm}
\label{sec:circular-shift-general}
    
    We now extend our approach from $n=3$ and $4$ to general $n$.  As noted above,
    we describe a deterministic strategy.  

    We call our strategy \basicstrat.
    Informally, the idea is as follows: once we get a value correct (that is,
    $\gamma_r(i) = \sigma(i)$), we never change it in subsequent rounds.  We take
    the misplaced values in $\gamma_r$ and cycle them around one step to the right
    to form
    $\gamma_{r+1}$.  Over successive guesses, each value travels to the correct
    position.

 \begin{definition1}[\basicstrat]\label{def:basicstrat}
Let $\gamma_1$ be the identity in $S_n$.  Suppose that, in round
$r$, we submit guess $\gamma_r$ and receive feedback $U_r$ as in \eqref{eq:def_Ur}.
If $U_r = [n]$, we are done in $r$
rounds. Otherwise, we define $\gamma_{r+1}$ as follows:
write $[n] \setminus U_r = \{a_1, \ldots, a_{\ell}\}$ with $a_j < a_{j+1}$,
and assign
$$
\gamma_{r+1}(i) := \begin{cases}
\gamma_r(i) & \text{if $i \in U_r$,} \\
\gamma_r(a_{j-1}) & \text {if $i = a_j$ for some $j > 1$,} \\
\gamma_r(a_\ell) & \text {if $i = a_1$.}
\end{cases}
$$
 \end{definition1}

 The approach depicted in Figures~\ref{fig-n3} and~\ref{fig-n4} is \basicstrat.  For a larger example, see Figure~\ref{fig-n9}.

 \begin{figure}
   \begin{center}
  \begin{wordle}[present color=WordleAbsent]{724853169}
    123456789
    821354679
    728153469
    724853169
  \end{wordle}
   \end{center}
   \caption{The deterministic strategy \basicstrat guessing the permutation
$\sigma = (7, 2, 4, 8, 5, 3, 1, 6, 9)$.  The strategy
succeeds in four rounds.}
     \label{fig-n9}
 \end{figure}

 We conjecture that \basicstrat is optimal in the following sense: Suppose that the
 setter chooses $\sigma \in S_n$ uniformly at random.
Given a strategy ${\mathcal S}$,
for each $r$, let $P_r({\mathcal S})$ denote the probability that ${\mathcal S}$
guesses the correct permutation within $r$ guesses.

\begin{conjecture}\label{conj:optimal}
$P_r(\basicstrat)$ attains the maximal value of $P_r$ for all $r \ge 1$.
\end{conjecture}
We have verified the conjecture for $n \le 6$.
See Section~\ref{sec:open} for further discussion.

Although we cannot prove that \basicstrat is optimal, we can completely analyze its cost.
To do so, we introduce some numbers that often arise in the combinatorics of
permutations.

\section{Eulerian numbers}
\label{sec:eulerian}

One could write a whole book on the Eulerian numbers~\cite{petersen}.  One
common definition of the Eulerian numbers
is in terms of the number of {\emph{excedances}} in a permutation:

\begin{definition1}\label{def:exc}
For $\pi \in S_n$, an \emph{excedance} is a value $i$ for which $\pi(i) > i$.  Let
$\Exc(\pi)$ denote the set of excedances: $\Exc(\pi) := \{i \mid \pi(i) > i\}$.
Write $\exc(\pi) := \#\Exc(\pi)$, the number of excedances of $\pi$.
\end{definition1}
\begin{definition1}\label{def:euler}
The \emph{Eulerian number} $A(n,k)$ is $\#\{\pi \in S_n \mid \exc(\pi) = k\}$.
\end{definition1}

For example, when $n = 3$, the only permutation with zero excedances is the identity, so $A(3,0) = 1$.
The only permutation in $S_3$ with two excedances is $(2,3,1)$, so $A(3,2) = 1.$  We deduce that $A(3,1) = 4$.

It is useful to write down a recurrence for the Eulerian numbers.

\begin{proposition}\label{prop:excrecur}
The Eulerian numbers $A(n,k)$ satisfy the recurrence
\begin{equation}\label{eq:eulerrecur}
A(n,k) = (k+1)A(n-1,k) + (n-k)A(n-1,k-1)
\end{equation}
with initial conditions $A(0,0) = 1$ and $A(0,k) = 0$ for $k \ne 0$.
\end{proposition}

\begin{proof}
We induct on $n$.  The unique element of $S_0$ has $0$ excedances, verifying the
initial condition.

Every permutation $\pi \in S_n$ can be uniquely decomposed in the form $\alpha \beta$, where
$\beta$ fixes $n$ and $\alpha$ is either the identity, or swaps $n$ with some value $j$.
We view $\beta$ as an element of $S_{n-1}$.  

How can $\pi$ have exactly $k$ excedances?  There are two possibilities. First, $\beta$ could 
have $k$ excedances, and $\alpha$ is either the identity, or swaps $n$ with one of the $k$ places in $\Exc(\beta)$. Second, $\beta$ could have $k-1$
excedances, and $\alpha$ swaps $n$ with one of the $n-k$ places
not in $\Exc(\beta)$. This corresponds to the desired recurrence.
\end{proof}

The recurrence in Proposition \ref{prop:excrecur} implies the reflective symmetry $A(n,k) = A(n,n-k-1)$.

There is also a well-known generating function representation of the Eulerian numbers,
which relates to Euler's work on the subject:
\begin{equation}\label{eq:euler-gf}
\frac{\sum_k A(n,k) t^k}{(1-t)^{n+1}} = \sum_j (j+1)^n t^j.
\end{equation}
One way to derive \eqref{eq:euler-gf} is to start with the infinite sequence whose enumerator
is the right hand side of \eqref{eq:euler-gf} and then apply the finite difference operator $n+1$ times.
The result is a sequence with a polynomial enumerator, that is, a finite number of nonzero elements.
As shown for $n = 3$ in Figure~\ref{fig-fd}, the nonzero elements of the final sequence are the Eulerian numbers.

\begin{figure}
      \begin{center}
      \begin{tikzpicture}[xscale=0.433,yscale=.5]
        \node at (-12,0) {$\cdots$};
        \node at (-10,0) {0};
        \node at (-8,0) {0};
        \node at (-6,0) {0};
        \node at (-4,0) {0};
        \node at (-2,0) {0};
        \node at (0,0) {0};
        \node at (2,0) {1};
        \node at (4,0) {8};
        \node at (6,0) {27};
        \node at (8,0) {64};
        \node at (10,0) {125};
        \node at (12,0) {$\cdots$};
        \node at (-11,-1) {$\cdots$};
        \node at (-9,-1) {0};
        \node at (-7,-1) {0};
        \node at (-5,-1) {0};
        \node at (-3,-1) {0};
        \node at (-1,-1) {0};
        \node at (1,-1) {1};
        \node at (3,-1) {7};
        \node at (5,-1) {19};
        \node at (7,-1) {37};
        \node at (9,-1) {61};
        \node at (11,-1) {$\cdots$};
        \node at (-10,-2) {$\cdots$};
        \node at (-8,-2) {0};
        \node at (-6,-2) {0};
        \node at (-4,-2) {0};
        \node at (-2,-2) {0};
        \node at (0,-2) {{\color{red}1}};
        \node at (2,-2) {6};
        \node at (4,-2) {12};
        \node at (6,-2) {18};
        \node at (8,-2) {24};
        \node at (10,-2) {$\cdots$};
        \node at (-9,-3) {$\cdots$};
        \node at (-7,-3) {0};
        \node at (-5,-3) {0};
        \node at (-3,-3) {0};
        \node at (-1,-3) {{\color{red}1}};
        \node at (1,-3) {{\color{red}5}};
        \node at (3,-3) {6};
        \node at (5,-3) {6};
        \node at (7,-3) {6};
        \node at (9,-3) {$\cdots$};
        \node at (-8,-4) {$\cdots$};
        \node at (-6,-4) {0};
        \node at (-4,-4) {0};
        \node at (-2,-4) {{\color{red}1}};
        \node at (0,-4) {{\color{red}4}};
        \node at (2,-4) {{\color{red}1}};
        \node at (4,-4) {0};
        \node at (6,-4) {0};
        \node at (8,-4) {$\cdots$};
    \end{tikzpicture}
    \end{center}
  \caption{We start with the sequence whose enumerator is $\sum_j (j+1)^n t^j$
    (for $n = 3$) and apply finite differences (that is, multiply by $1-t$) four times.}
  \label{fig-fd}
  \end{figure}
Another common definition of the Eulerian numbers
is in terms of {\emph{descents}}.  A {\emph{descent}} of a
permutation is a place $i$ with $\pi(i) > \pi(i+1)$; the number of permutations in
$S_n$ with exactly $k$ descents is also given by $A(n,k)$.  One way to prove the equivalence of
this alternative formulation for $A(n,k)$ is 
verify that these numbers satisfy the same recurrence~\eqref{eq:eulerrecur}.
Foata and Sch\"utzenberger~\cite{FoSc} give an explicit bijection between permutations
with $k$ excedances and those with $k$ descents.

The Eulerian numbers show up in many other contexts.
For example, let $Y_1, \ldots, Y_n$ be independent uniform random
variables in the interval $(0,1)$, and define $X = \floor{\sum_{i=1}^n Y_i}$. The
probability that $X = k$ is $A(n,k)/n!$.  We now show that these numbers
arise in the study of \classic.

\mysubsection{The main theorem}
\label{sec:main-thm}

We will prove the following lemma, which connects \basicstrat to excedances:

\begin{lemma}\label{lemma:one-perm}
The number of guesses it takes \basicstrat to guess a permutation $\sigma$ is
$1 + \exc(\sigma)$.
\end{lemma}

Lemma~\ref{lemma:one-perm} immediately implies our main theorem:
\begin{theorem}\label{thm:main}
The number of permutations solved in $k+1$ guesses by \basicstrat is the
Eulerian number $A(n,k)$.
\end{theorem}

Since the Eulerian numbers have reflexive symmetry, we get a simple description
of the average-case performance of \basicstrat:
\begin{corollary}
  If $\sigma \in S_n$ is chosen uniformly at random, then the expected number
  of rounds for \basicstrat is $(n+1)/2$.
\end{corollary}

This leads us to a weaker version of Conjecture~\ref{conj:optimal}:
\begin{conjecture}\label{conj:optimal-weak}
For any strategy for \classic, the average number of rounds is at least $(n+1)/2$.
\end{conjecture}

It remains to prove Lemma~\ref{lemma:one-perm}.  The idea is to keep track of
the sets $R_k$ of values placed further to the right by $\gamma_k$ than by
$\sigma$:
$$R_k := \{ \gamma_k(i) : \gamma_k(i) > \sigma(i) \}.$$
We will show that $\#R_k$ decreases by one each round, starting with $\#R_1 = \exc(\sigma)$,
until the final round where $R_{1+\exc(\sigma)} = \emptyset$ and
$\gamma_{1+\exc(\sigma)} = \sigma$.  This is illustrated in Figure~\ref{fig-n9-circles}.

 \begin{figure}
   \begin{center}
  \begin{wordle}[present color=WordleAbsent]{724853169}
          [{
              \draw[very thick,color=highlightcircle] (W-1-4) circle (.4cm);
              \draw[very thick,color=highlightcircle] (W-1-7) circle (.4cm);
              \draw[very thick,color=highlightcircle] (W-1-8) circle (.4cm);
              \draw[very thick,color=highlightcircle] (W-2-6) circle (.4cm);
              \draw[very thick,color=highlightcircle] (W-2-8) circle (.4cm);
              \draw[very thick,color=highlightcircle] (W-3-7) circle (.4cm);
          }]
    123456789
    821354679
    728153469
    724853169
  \end{wordle}
   \end{center}
   \caption{The deterministic strategy \basicstrat guessing the permutation
     $\sigma = (7, 2, 4, 8, 5, 3, 1, 6, 9)$.  We highlight the values in $R_k$ for each $k$; that is,
     those values that $\gamma_k$ places to the right of $\sigma$.}
     \label{fig-n9-circles}
 \end{figure}

\begin{proof}[Proof of Lemma~\ref{lemma:one-perm}]
Let $\sigma$ be the setter's permutation.  Run \basicstrat, producing a series of
guesses $\gamma_1$, \ldots, $\gamma_r$ for some $r$ with $\gamma_r = \sigma$.  Recall that
$U_k$ is the set of positions where $\gamma_k$ agrees with $\sigma$ and 
$R_k$ is the set of values that $\gamma_k$ places further right than $\sigma$ does.
(The values in $R_k$ are highlighted in Figure~\ref{fig-n9-circles}.)
Let $L_k$ be the misplaced values that are not in $R_k$:
the values that $\gamma_k$ places further to the left than $\sigma$ does.

Observe that $i \in R_1$ if and only if $i > \sigma^{-1}(i)$, meaning 
$\sigma^{-1}(i) \in \Exc(\sigma)$.  We see that $\#R_1 = \exc(\sigma).$

Assuming $U_k \neq [n]$, how does $R_{k+1}$ differ from $R_k$?
The set $[n]\setminus U_k$ has at least two elements,
because it is impossible to have exactly one number misplaced.
Write $[n]\setminus U_k = \{a_1, \ldots, a_\ell\}$ with $a_j < a_{j+1}$.
Since $\gamma_k(a_\ell)$ is misplaced and it is the rightmost misplaced number,
$\gamma_k(a_\ell) \in R_k$.
Following \basicstrat,
any values in the correct positions remain in the correct positions:
for $i\in U_k$, including $1 \leq i < a_1$, we have $\gamma_{k+1}(i) = \gamma_{k}(i) = \sigma(i)$.
Also following \basicstrat, the rightmost misplaced value moves
to the position of the leftmost misplaced value:
$\gamma_{k+1}(a_1) = \gamma_{k}(a_\ell)$.
So $\gamma_k(a_\ell) \not\in R_{k+1}$.

Any other values in $R_k$ move further right, and are also in $R_{k+1}$.  
Each value in $L_k$ moves right, but not beyond its location
in $\sigma$, so it is not in $R_{k+1}$.  We conclude that
$R_{k+1} = R_k \setminus \{a_\ell\}$, so $\#R_{k+1} = \#R_k - 1$.

\basicstrat terminates when $U_r = [n]$, meaning $\gamma_r = \sigma.$
This happens exactly when $R_r$ is empty, so
the number of rounds for \basicstrat is $1 + \#R_1 = 1 + \exc(\sigma)$.
\end{proof}

\section{Multicolor Permutation Wordle}
\label{sec:rainbow}

We now generalize from the ``monochrome'' case
to a ``multicolor'' version of the problem.  Imagine that, instead of
just choosing a permutation of $[n]$, the setter chooses an arrangement
of $n$ cards from a larger deck: each card has a rank between $1$ and $n$ and
a color.  We still require that the setter's arrangement contain each rank exactly once.
More formally:
\begin{definition1}
   A \emph{multicolor permutation} is
a permutation in $S_n$ together with a coloring $c \colon [n] \to \{0, 1, \ldots\}$.  An
\emph{$(s_1, \ldots, s_n)$-colored permutation} is one where $0 \le c(i) < s_i$ for all $i$,
and an $s$-colored permutation is one where $0 \le c(i) < s$.
  \end{definition1}

We focus on these cases with explicit bounds on $c(i)$ because the notion of a uniform random
$(s_1, \ldots, s_n)$-colored or $s$-colored permutation is well defined.  We will see that the analysis is
roughly the same whether we have $n$ parameters $(s_1, \ldots, s_n)$ or a single
parameter $s$, but the statements are simpler for a single $s$.

In \rainbow, the setter chooses a multicolor permutation $(\sigma, \kappa)$.
The feedback for guess $(\gamma, c)$ is
the set of positions $i$ where the guess agrees the setter's $\sigma$ and $\kappa$,
that is, $\gamma(i) = \sigma(i)$ and $c(\gamma(i)) = \kappa(\sigma(i))$.

\begin{remark1}
One could reasonably ask whether this is the right rule for the feedback.  The honest
answer is that this is the rule where the analysis of \basicstrat seems to lead to
interesting numbers.  For example, one interpretation of the Wordle rules would be to
get feedback of the form ``Red 1 in the right place; blue 2 is misplaced; there is no red 3.''
The two-color game with this more informative feedback rule reveals
which numbers are red and which blue after just one guess;
it is too close to the original ``monochrome''
\classic problem.
\end{remark1}

As in the monochrome version, the argument of Section~\ref{sec:avg-case} applies: we may
assume that the setter chooses $\sigma$ and $\kappa$ uniformly at random, and we may
employ a deterministic strategy.  Our first guess is $(\gamma_1, c_1)$ where $\gamma_1$ is the
identity permutation and $c_1(i) = 0$ for all $i$.

\mysubsection{\basicstrat for \rainbow}
\label{sec:rainbow-circular}

We now generalize from monochrome \basicstrat to the multicolor setting.
We rotate the values of $\gamma_r$
just as before, and we start with all colors equal to $0$.  When a
value with color $a$ loops all the way around to its starting position, that means the
color is incorrect, so we increment the color to $a+1$ as the value $i$ begins its next circuit.

\begin{definition1}[\basicstrat for \rainbow] Let \linebreak
$\gamma_1$ be the identity in $S_n$, and let $c_1 \equiv 0$.
Suppose that, in round
$r$, we submit guess $(\gamma_r, c_r)$ and receive feedback
$$U_r = \{i \mid \gamma_r(i) = \sigma_r(i)\ \mathrm{and}\ c_r(\gamma_r(i)) = \kappa(\sigma(i))\},$$
the set of positions $i$ where $(\gamma_r, c_r)$ and $(\sigma, \kappa)$ agree on both
the value and the color of that value.
If $U_r = [n]$, we are done in $r$
rounds. Otherwise, we define $\gamma_{r+1}$ and $c_{r+1}$ as follows:
write $[n] \setminus U_r = \{a_1, \ldots, a_{\ell}\}$ with $a_j < a_{j+1}$. Let
$$
\gamma_{r+1}(i) := \begin{cases}
\gamma_r(i) & \text{if $i \in U_r$,} \\
\gamma_r(a_{j-1}) & \text {if $i = a_j$ for some $j > 1$,} \\
\gamma_r(a_\ell) & \text {if $i=a_1$.}
\end{cases}
$$
Also,
$$
c_{r+1}(i) := \begin{cases}
c_r(i) + 1 & \text{$i = \gamma_{r+1}(a_j)$ for some $j > 1$ and $\gamma_{r+1}(a_{j-1}) < a_j \le \gamma_{r+1}(a_j)$,} \\
c_r(i) + 1 & \text{$i = \gamma_{r+1}(a_1)$ and either $a_1 \le \gamma_{r+1}(a_1)$ or $\gamma_{r+1}(a_\ell) < a_1$,} \\
c_r(i) & \text{otherwise}.
\end{cases}
$$
\end{definition1}

We next prove a multicolor analogue of Lemma~\ref{lemma:one-perm}.

\begin{lemma}\label{lemma:one-perm-multi}
The number of guesses it takes \basicstrat to guess a multicolor permutation $(\sigma, \kappa)$
is
$$
1 + \exc(\sigma) + \sum_{i=1}^n \kappa(i).
$$
\end{lemma}

The idea of the proof is to count down the number of rounds remaining.
We introduce quantities $Z_{k,i}$ representing the number of times after
round $k$ that the value $i$ will wrap around before we finish solving.
In each round, one value ($i = \gamma_k(a_\ell)$) wraps around, so the total
$\sum_i Z_{k,i}$ decreases by $1$. We are done when the countdown reaches zero.  See
Figure~\ref{fig-n5-multi} for an example.

\begin{figure}[h]
  \begin{center}
          \begin{wordle}[noletters]{12345}
        [{
            \node at (W-1-1) {\Large\bfseries\sffamily\color{multizero} 1};
            \node at (W-1-2) {\Large\bfseries\sffamily\color{multizero} 2};
            \node at (W-1-3) {\Large\bfseries\sffamily\color{multizero} 3};
            \node at (W-1-4) {\Large\bfseries\sffamily\color{multizero} 4};
            \node at (W-1-5) {\Large\bfseries\sffamily\color{multizero} 5};
            \node at (W-2-1) {\Large\bfseries\sffamily\color{multizero} 5};
            \node at (W-2-2) {\Large\bfseries\sffamily\color{multizero} 1};
            \node at (W-2-3) {\Large\bfseries\sffamily\color{multizero} 2};
            \node at (W-2-4) {\Large\bfseries\sffamily\color{multizero} 3};
            \node at (W-2-5) {\Large\bfseries\sffamily\color{multizero} 4};
            \node at (W-3-1) {\Large\bfseries\sffamily\color{multizero} 3};
            \node at (W-3-2) {\Large\bfseries\sffamily\color{multizero} 5};
            \node at (W-3-3) {\Large\bfseries\sffamily\color{multizero} 1};
            \node at (W-3-4) {\Large\bfseries\sffamily\color{multizero} 2};
            \node at (W-3-5) {\Large\bfseries\sffamily\color{multizero} 4};
            \node at (W-4-1) {\Large\bfseries\sffamily\color{multizero} 3};
            \node at (W-4-2) {\Large\bfseries\sffamily\color{multione} 2};
            \node at (W-4-3) {\Large\bfseries\sffamily\color{multizero} 5};
            \node at (W-4-4) {\Large\bfseries\sffamily\color{multizero} 1};
            \node at (W-4-5) {\Large\bfseries\sffamily\color{multizero} 4};
            \node at (W-5-1) {\Large\bfseries\sffamily\color{multizero} 3};
            \node at (W-5-2) {\Large\bfseries\sffamily\color{multione} 2};
            \node at (W-5-3) {\Large\bfseries\sffamily\color{multione} 1};
            \node at (W-5-4) {\Large\bfseries\sffamily\color{multizero} 5};
            \node at (W-5-5) {\Large\bfseries\sffamily\color{multizero} 4};
            \node at (W-6-1) {\Large\bfseries\sffamily\color{multizero} 3};
            \node at (W-6-2) {\Large\bfseries\sffamily\color{multione} 2};
            \node at (W-6-3) {\Large\bfseries\sffamily\color{multione} 5};
            \node at (W-6-4) {\Large\bfseries\sffamily\color{multione} 1};
            \node at (W-6-5) {\Large\bfseries\sffamily\color{multizero} 4};
            \draw[very thick,color=highlightcircle] (W-1-1) circle (.4cm);
            \draw[very thick,color=highlightcircle] (W-1-2) circle (.4cm);
            \draw[very thick,color=highlightcircle] (W-1-3) circle (.4cm);
            \draw[very thick,color=highlightcircle] (W-1-5) circle (.4cm);
            \draw[very thick,color=highlightcircle] (W-1-5) circle (.3cm);
            \draw[very thick,color=highlightcircle] (W-2-1) circle (.4cm);
            \draw[very thick,color=highlightcircle] (W-2-2) circle (.4cm);
            \draw[very thick,color=highlightcircle] (W-2-3) circle (.4cm);
            \draw[very thick,color=highlightcircle] (W-2-4) circle (.4cm);
            \draw[very thick,color=highlightcircle] (W-3-2) circle (.4cm);
            \draw[very thick,color=highlightcircle] (W-3-3) circle (.4cm);
            \draw[very thick,color=highlightcircle] (W-3-4) circle (.4cm);
            \draw[very thick,color=highlightcircle] (W-4-3) circle (.4cm);
            \draw[very thick,color=highlightcircle] (W-4-4) circle (.4cm);
            \draw[very thick,color=highlightcircle] (W-5-4) circle (.4cm);
        }]
    xxxxx
    xxxx5
    1xxx5
    12xx5
    12xx5
    12345
  \end{wordle}
  \end{center}
  \caption {Applying \basicstrat to two-color permutation Wordle with $n=5$.
    Here color $0$ is shown in \multizeroname,
    and color $1$ in \multionename.  The highlights show how often
    each number will need to wrap around; note that the $5$ is highlighted twice, once
    because of its location and once because of its color.}
  \label{fig-n5-multi}
 \end{figure}
\begin{proof}
As in the proof of Lemma~\ref{lemma:one-perm}, we construct a quantity which decreases by $1$ in
each round and which is $0$ when we have the right answer.
Suppose that, in round $k$, we have some guess $(\gamma_k, c_k)$.
For each $i$, let $\delta(i) = 1$ when $i$ occurs further to left in $\sigma$ than in the identity;
that is, $i > \sigma^{-1}(i)$, or $\sigma^{-1}(i)$ is an
excedance of $\sigma$. Let $\epsilon_k(i) = 1$ when $i > \gamma_k^{-1}(i)$.
Let $$Z_{k,i} = \kappa(i) - c_k(i) + \delta(i) - \epsilon_r(i)$$
and let $Z_k = \sum_{i=1}^n Z_{k,i}$.
We show that $Z_1 = \exc(\sigma) + \sum_i \kappa(i),$
that $(\gamma_k, c_k) \neq (\sigma, \kappa)$ implies $Z_k > 0,$
and that $Z_{k+1} = Z_k-1$, as illustrated by Figure~\ref{fig-n5-multi}.

Since $\gamma_1$ is the identity and $c_1 \equiv 0$,
we see that $Z_{1,i} \geq 0$ for all values $i$ and $Z_1 = \exc(\sigma) + \sum_i \kappa(i)$.
Note also that 
$\kappa(i) \geq c_k(i)$ for all $k$ and $i$, because, according to \basicstrat,
$c_k(i)$ is a nondecreasing sequence in $k$, and $c_{k+1}(i) = c_k(i)+1$ implies that
value $i$ with color $c_k(i)$ visited every position in $[n] \setminus U_k$ in earlier rounds.
When $\gamma_k^{-1}(i) \in U_k,$, we have $\delta(i) = \epsilon_k(i)$ and $\kappa(i) = c_k(i)$,
so $Z_{k,i}= Z_{k+1,i} = 0.$
When $U_k = [n]$, $\gamma_k = \sigma$, $c_k = \kappa$, and $Z_k = 0$.

Suppose $U_k \subsetneq [n].$
Enumerate $[n] \setminus U_k$ as an increasing sequence $a_1, \ldots, a_\ell$.
How do $Z_{k,i}$ and $Z_{k+1,i}$ compare when $i = \gamma_k(a_j)$ for $j < \ell$?
If $a_j < i \leq a_{j+1}$, then $c_{k+1}(i) = c_k(i)+1$, $\epsilon_k(i) = 0$, and $\epsilon_{k+1}(i) = 1$.
Otherwise, $c_{k+1}(i) = c_k(i)$ and $\epsilon_{k+1}(i) = \epsilon_k(i)$.  Either way, $Z_{k+1,i} = Z_{k,i}.$

It remains to consider $Z_{k,i}$ and compare $Z_{k+1, i}$ and $Z_{k, i}$ for $i = \gamma_k(a_\ell)$.
There are three possibilities.
If $i > a_\ell \ge a_1$, then $\epsilon_{k+1}(i) = \epsilon_k(i) = 1$ and $c_{k+1}(i) = c_k(i) + 1$.
If $a_\ell \ge a_1 \ge i$, then $\epsilon_{k+1}(i) = \epsilon_k(i) = 0$ and $c_{k+1}(i) = c_k(i) + 1$.
If $a_\ell \ge i > a_1$, then $\epsilon_{k+1}(i) = 1,$ $\epsilon_k(i) = 0$, $c_{k+1}(i) = c_k(i)$,
and $c_k(i) < \kappa(i) + \delta(i).$
In each case, $Z_{k,i} > 0$ and $Z_{k+1,i} = Z_{k,i} - 1$.
Since this is the only $i$ for which $Z_{k+1,i} \ne Z_{k,i}$, we have $Z_{k+1} = Z_k - 1$.
The positive value of $Z_{k,i}$ verifies that $(\gamma_k, c_k) \neq (\sigma, \kappa)$ implies $Z_k > 0.$

At round $r = 1+\exc(\sigma) + \sum_i \kappa(i)$, we have $Z_r = 0$. and, consequently,
$(\gamma_r, c_r) = (\sigma, \kappa)$.

\end{proof}

\begin{conjecture}\label{conj:rainbow-optimal}
  For each $r$, the probability that \basicstrat solves a uniformly chosen
  $(s_1, \ldots, s_n)$-colored permutation within $r$ rounds
  is optimal, that is, it is at least as large as for any other strategy.
\end{conjecture}

\mysubsection{Multicolor recurrences and generating functions}~\label{sec:rainbow-rec-gf}
\label{sec:rainbow-gf}

To understand the expected behavior of \basicstrat on multicolor permutations with
various numbers of colors, it is convenient to work with {\emph{generating functions}}.
On Page~\pageref{eq:euler-gf}, we mentioned the standard representation of the
Eulerian numbers~\eqref{eq:euler-gf}, restated here:

\begin{equation*}
\frac{\sum_k A(n,k) t^k}{(1-t)^{n+1}} = \sum_j (j+1)^n t^j.
\end{equation*}

We generalize this identity to the ``Wordle numbers''.  We prove our results in
the general multicolor form, and then give the simpler expressions we get in the
$s$-colored case.

\begin{definition1}
We let $\Dmult{n}(n,k)$ denote the number of $(s_1, \ldots, s_n)$-colored permutations of
length $n$ for which \basicstrat succeeds in exactly $k+1$ rounds.
We let $W_s(n,r)$ denote the number of $s$-colored permutations for
which \basicstrat succeeds in exactly $r+1$ rounds.
\end{definition1}

\begin{theorem}\label{thm:multi-gf}
\begin{equation}\label{eq:multi-gf}
\frac{\sum_k \Dmult{n}(n,k) u^k}{(1-u)\prod_\ell(1-u^{s_\ell})} = \sum_j (j+1)^n u^j. 
\end{equation}
\end{theorem}
\begin{proof}
By Lemma~\ref{lemma:one-perm-multi},
\begin{align*}
\sum_k \Dmult{n}(n,k) u^k & = \sum_{\sigma, \kappa} u^{\exc(\sigma) + \kappa(1) + \cdots + \kappa(n)} \\
& = \left(\sum_{\sigma} u^{\exc(\sigma)}\right)
\left(\sum_{\kappa} \prod_{\ell=1}^n u^{\kappa(\ell)}\right) \\
& = \left(\sum_k A(n,k) u^k\right) \prod_{\ell=1}^n (1 + u + \cdots + u^{s_\ell - 1}),
\intertext{since $A(n,k)$ counts excedances. Using~\eqref{eq:euler-gf},}
\sum_k \Dmult{n}(n,k) u^k & = \left(\sum_j (j+1)^n u^j\right)(1-u)^{n+1}
\prod_{\ell=1}^n \frac{1-u^{s_\ell}}{1-u}.
\end{align*}
The result follows.
\end{proof}

From the generating function~\eqref{eq:multi-gf} we can derive a recurrence for
$\Dmult{n}(n,k)$.  We use $\noell$ to denote the
$(n-1)$-long sequence obtained from $(s_1, \ldots, s_n)$ by removing $s_\ell$.

\begin{theorem}\label{thm:multi-recur}
The counts $\Dmult{n}(n,k)$ satisfy the recurrence
\begin{equation}\label{eq:multi-recur}
\begin{multlined}
\Dmult{n}(n,k) = (k+1)\Dmult{n-1}(n-1,k)
+ \sum_{r=1}^{s_n-1}\Dmult{n-1}(n-1,k-r) \\ - k\Dmult{n-1}(n-1,k-s_n)
+ \sum_{\ell=1}^n s_\ell W_{\noell}(n-1,k-s_\ell),
\end{multlined}
\end{equation}
with initial conditions $W_{()}(0,0) = 1$ and $W_{()}(0,k) = 0$ for $k \ne 0$.
\end{theorem}

The proof uses standard generating function techniques; see Appendix~\ref{sec:gf-proofs}.

Theorems~\ref{thm:multi-gf} and~\ref{thm:multi-recur} are both simpler in the
$s$-colored case:
\begin{theorem}\label{thm:D-gf}
$$\frac{\sum_k W_s(n,k) u^k}{(1-u)(1-u^s)^n} = \sum_j (j+1)^n u^j.$$
\end{theorem}

\begin{theorem}\label{thm:D-recur}\label{thm:rainbow}
The counts $W_s(n,k)$ satisfy the recurrence
\begin{equation}
\label{eq:Drecur}
W_s(n,k) =  (k+1) W_s(n-1,k) + (sn-k) W_s(n-1,k-s) + \sum_{\ell=1}^{s-1} W_s(n-1,k-\ell)
\end{equation}
with initial conditions $W_s(0,0) = 1$ and $W_s(0,k) = 0$ for $k \ne 0$.
\end{theorem}
Note that, when $s = 1$, this expression ~\eqref{eq:Drecur}
is the recurrence~\eqref{eq:eulerrecur} for the Eulerian numbers.

The recurrence \eqref{eq:Drecur} has a natural symmetry:

\begin{corollary}\label{cor:rainbow-count}
We have $W_s(n,sn-r-1) = W_s(n,r)$.
\end{corollary}

\mysubsection{Sums of consecutive values}~\label{sec:relate-Eulerian}
\label{sec:rainbow-consecutive}

We can find patterns in the sums of consecutive values of these
sequences $W_s(n,r)$.  For example, consider adding pairs of numbers from
$W_2(n,r)$; the case $n=4$ is shown in Figure~\ref{fig:D24}.  If we start with the
first two numbers, then the next two, and so on, we obtain multiples of the
Eulerian numbers: the sequence is $2^n A(n,r)$.  However, if we instead start with
just one number, then two at a time, 
we obtain a different known combinatorial sequence: the Eulerian numbers
of type B~\cite{Chow-Gessel}. We now
develop a recurrence for these sums of consecutive values of Wordle numbers;
in Section~\ref{sec:higher}, we will show that these sums give the ``higher-order
Eulerian numbers''~\cite{steingrimsson}.

\begin{remark1}
  Given two sequences with the same sum,
  it is not surprising that one can write down a sequence that simultaneously
  refines both of them; that is, the numbers
  can add to one or the other depending on how we pair
them up.  What is surprising is that this particular
sequence $W_2(n,r)$ arises in its own
right, in the analysis of two-color \basicstrat.  One is naturally led to wonder what else
these $W_s$ numbers might be counting.
\end{remark1}

\begin{figure}
\begin{center}
\begin{tikzpicture}
\draw (0,0) node{1};
\draw (1,0) node{15};
\draw (2,0) node{61};
\draw (3,0) node{115};
\draw (4,0) node{115};
\draw (5,0) node{61};
\draw (6,0) node{15};
\draw (7,0) node{1};
\draw [decorate,decoration={brace,amplitude=5pt,raise=2ex}] (-.1,0) -- (1.2,0) node[midway,yshift=4.5ex]{16};
\draw [decorate,decoration={brace,amplitude=5pt,raise=2ex}] (1.8,0) -- (3.3,0) node[midway,yshift=4.5ex]{176};
\draw [decorate,decoration={brace,amplitude=5pt,raise=2ex}] (3.7,0) -- (5.2,0) node[midway,yshift=4.5ex]{176};
\draw [decorate,decoration={brace,amplitude=5pt,raise=2ex}] (5.8,0) -- (7.1,0) node[midway,yshift=4.5ex]{16};
\draw (0,0) node[yshift=-4.5ex]{1};
\draw (0,-.3) -- (0,-.55);
\draw [decorate,decoration={brace,amplitude=5pt,mirror,raise=2ex}] (0.8,0) -- (2.2,0) node[midway,yshift=-4.5ex]{76};
\draw [decorate,decoration={brace,amplitude=5pt,mirror,raise=2ex}] (2.7,0) -- (4.3,0) node[midway,yshift=-4.5ex]{230};
\draw [decorate,decoration={brace,amplitude=5pt,mirror,raise=2ex}] (4.8,0) -- (6.2,0) node[midway,yshift=-4.5ex]{76};
\draw (7,0) node[yshift=-4.5ex]{1};
\draw (7,-.3) -- (7,-.55);
\end{tikzpicture}
\end{center}
\caption{Two ways to group the sequence $W_2(4,k)$.}\label{fig:D24}
\end{figure}

\begin{definition1}
For each $s, n, k$, we define $F_s(n,k)$ to be the sum of $s$ consecutive values of $W_s(n,r)$:
$$
F_s(n,k) = \sum_{\ell=0}^{s-1} W_s(n,k-\ell).
$$
\end{definition1} 

\begin{theorem}\label{thm:F-gf}
\begin{equation}\label{eq:F-gf}
\frac{\sum_k F_s(n,k) u^k}{(1-u^s)^{n+1}} = \sum_j (j+1)^n u^j.
\end{equation}
\end{theorem}

\begin{proof}
From the definition, $\sum_k F_s(n,k) u^k = (1 + \cdots + u^{s-1}) \sum_k W_s(n,k) u^k$.
The result follows immediately from Theorem~\ref{thm:D-gf}.
\end{proof}

\begin{theorem}\label{thm:sums}\label{thm:F-recur}
The numbers $F_s(n,k)$ satisfy the recurrence
$$
F_s(n,k) = (k+1)F_s(n-1,k) + (s(n+1)-k-1) F_s(n-1,k-s).
$$
with initial conditions $F_s(0,k) = 1$ when $0 \le k < s$ and $F_s(0,k) = 0$ otherwise.
\end{theorem}

Again, a proof can be found in Appendix~\ref{sec:gf-proofs}.

\begin{remark1}\label{rem:log-concave-D-F}
  The sequences $\Dmult{n}(n,\cdot)$ and $F_s(n,\cdot)$ are log-concave.
Note that \linebreak $\Dmult{n}(n,\cdot)$ is the iterated convolution of
  $A(n,\cdot)$ with sequences $(1, \ldots, 1)$ of length $s_\ell$ for
$1\leq \ell \leq n$, and $F_s(n,\cdot)$ is the convolution of $W_s(n,\cdot)$
with $(1, \ldots, 1)$.  The
convolution of two log-concave sequences is log-concave~\cite{KeGe}.
\end{remark1}  

\section{Higher-order Eulerian numbers}
\label{sec:higher}
\label{sec:higher-order}

\mysubsection{Multiples of Eulerian numbers}
\label{sec:higher-e1}

As noted above, sums of consecutive values of $W_2(n,r)$ recover
the values $2^n A(n,r)$.
We begin our discussion
of higher-order Eulerian numbers by generalizing this observation to arbitrary $s$.

\begin{proposition}\label{prop:sums-euler-A}
For all $n$ and $k$,
$$
F_s(n,sk+s-1) = s^n A(n,k).
$$
\end{proposition}

\begin{proof}
For $n = 0$, $F_s(n,sk+s-1) = 1$ when $k=0$ and $F_s(n,sk+s-1) = 0$ otherwise.

For $n > 0$, by Theorem~\ref{thm:sums},
\begin{align*}
F_s(n,sk+s-1) & = (sk+s)F_s(n-1,sk+s-1) + (s(n+1)-sk-s) F_s(n-1,sk-1) \\
& = s\left[(k+1)F_s(n-1,sk+s-1) + (n-k)F_s(n-1,s(k-1)+s-1)\right].
\end{align*}
By Proposition~\ref{prop:excrecur}, this same recurrence is satisfied by $s^n A(n,k)$.
\end{proof}

\begin{corollary}\label{cor:sums}
The probability that \basicstrat succeeds on $s$-colored \classic in at most
$sr$ rounds is the same as the probability that \basicstrat succeeds on \classic in at most $r$ rounds.
\end{corollary}

\mysubsection{$z$-excedances}~\label{sec:higher:zexc}
\label{sec:higher-z-excedances}

To describe the other possible sums of $W_s(n,r)$, we need to introduce
the higher-order Eulerian numbers. We follow (and slightly generalize)
the approach of Steingr{\'\i}msson~\cite{steingrimsson}.
Let $S_{n,s}$ denote the set of $s$-colored permutations of length $n$.

\begin{definition1}
For $(\pi, \notkappa) \in S_{n,s}$, a \emph{$z$-excedance} is a value $j$ for which
either $\pi(j) > j$, or for which $\pi(j) = j$ and $\notkappa(\pi(j)) \ge z$.  Equivalently, it is
a value $j$ where $(\pi(j), \notkappa(\pi(j))) \succ (i, z-1)$ lexicographically.

For $1 \le z \le s$, we let $\Esz(n,k)$ denote the number of $s$-colored permutations of length $n$ with $k$ $z$-excedances. 
 \end{definition1}

Steingr{\'\i}msson analyzes the case $z=1$ and derives the recurrence
$$
E_{s,1}(n,k) = (sk+1)E_{s,1}(n-1,k) + (s(n-k)+s-1)E_{s,1}(n-1,k-1)
$$
and the generating function
\begin{equation}\label{eq:ES-gf}
\frac{\sum_k E_{s,1}(n,k) t^k}{(1-t)^{n+1}} = \sum (sj+1)^n t^j.
\end{equation}
For $s = 2$, these are the Eulerian numbers of type B.

We generalize Steingr{\'\i}msson's results to larger values of $z$:

\begin{theorem}\label{thm:zexcrecur}
The numbers $\Esz(n,k)$ satisfy the recurrence
$$
\Esz(n,k) = (sk+z)\Esz(n-1,k) + (s(n-k)+s-z)\Esz(n-1,k-1)
$$
with initial conditions $\Esz(0,k) = 1$ when $k=0$ and $\Esz(0,k)=0$ for $k \ne 0$.
\end{theorem}

The proof follows the same lines as that of Proposition~\ref{prop:excrecur}.
\begin{proof}
We induct on $n$; the unique element of $S_0$ has $0$ $z$-excedances,
verifying the initial condition.

Every $s$-colored permutation $(\pi, \notkappa)\in S_{n, s}$ can be uniquely constructed as follows:
start with an $s$-colored permutation $(\beta, \notlambda) \in S_{n-1, s}$.
Extend $\notlambda$ to $\notkappa$ by $\notkappa(i) = \notlambda(i)$ for $i < n$ and
making a choice $\notkappa(n)$ of a color for value $n$.
To make $\pi$, extend $\beta$ to $S_n$ by making $n$ a fixed point, then let
$\pi= \alpha\beta$ with $\alpha$ either the identity or the swap of $i$ and $n$ for some $i <n$.

How can $(\pi, \notkappa)$ have exactly $k$ $z$-excedances?  There are two possibilities. First,
$(\beta, \notlambda)$ could have $k$ $z$-excedances, and either $\alpha$ swaps $n$ with one of the
$k$ $z$-excedances (in which case any color $\notkappa(n)$ works), or $\alpha$ leaves $n$ in place
and $\notkappa(n) < z$. Given $(\beta, \notlambda)$ with $k$ $z$-excedances, this yields $sk + z$
$s$-colored permutations $(\pi, \notkappa)$.

Second, $(\beta, \notlambda)$ could have $k-1$ $z$-excedances, and either $\alpha$ swaps $n$ with one
of the $n-k$ non-$z$-excedances (in which case any color $\notkappa(n)$ works), or $\alpha$ leaves
$n$ in place and $\notkappa(n) \ge z$ (giving one more $z$-excedance). Given 
$(\beta, \notlambda)\in S_{n-1,s}$ with $k-1$ $z$-excedances, this yields $s(n-k) + s-z$
$s$-colored permutations $(\pi,\notkappa)\in S_{n,s}$.

The sum of these two cases yields the desired recurrence.
\end{proof}

Note that $E_{s,s}(n,k)$ is simply $s^n A(n,k)$; since no color can exceed $s-1$,
the colors play no role and we simply count excedances.  When $z=s$, the
recurrence of Theorem~\ref{thm:zexcrecur} simplifies to
$$
E_{s,s}(n,k) = s(k+1)E_{s,s}(n-1,k) + s(n-k)E_{s,s}(n-1,k-1), 
$$
which is the recurrence for $F_s(n,sk+s-1)$ discussed in
Proposition~\ref{prop:sums-euler-A}.  We can apply similar logic to any value of $z$.

\begin{proposition}\label{prop:sums-euler-B}
For $1 \le z \le s$, we have
$$
F_s(n,sk+z-1) = \Esz(n,k).
$$
\end{proposition}

\begin{proof}
When $n=0$, $F_s(n,sk+z-1)$ is $1$ when $k=0$ and $0$ otherwise, as desired.

For $n > 0$, by Theorem~\ref{thm:sums},
$$
F_s(n,sk+z-1) = (sk+z)F_s(n-1,sk+z-1) + (s(n-k)+s-z)F_s(n-1,s(k-1)+z-1).
$$
By Proposition~\ref{thm:zexcrecur}, this is the recurrence satisfied by
$\Esz(n,k)$.
\end{proof}

\begin{proposition}\label{prop:divisibility}
$$
E_{ds,dz}(n,k) = d^n \Esz(n,k).
$$
\end{proposition}

\begin{proof}
The proof is the same as that of Proposition~\ref{prop:sums-euler-A}, except that we
pull out a factor of $d$, rather than $s$, each time we apply the recurrence.
\end{proof}

\begin{remark1}\label{rem:log-concave-Esz}
Since $\Esz(n,\cdot)$ is a subsequence of the log-concave $F(n,\cdot)$ (see
Remark~\ref{rem:log-concave-D-F}),  $\Esz(n,\cdot)$ is also log-concave.
\end{remark1}

\begin{theorem}\label{thm:Esz-gf}
For $1 \le z \le s$,
$$\frac{\sum_k \Esz(n,k)t^k}{(1-t)^{n+1}} = \sum_j(sj+z)^n t^j.$$
\end{theorem}

Note that, when $z=1$, Theorem~\ref{thm:Esz-gf}
recovers Steingr{\'\i}msson's \eqref{eq:ES-gf}.
\begin{proof}
By Theorem~\ref{thm:F-gf},
$$
\sum_k F_s(n,k) u^k = (1-u^s)^{n+1} \sum_j (j+1)^n u^j.
$$
On each side, consider only terms corresponding to $u^k$ with $k \equiv z-1 \pmod s$:
$$
\sum_k F_s(n,sk+z-1) u^{sk+z-1} = (1-u^s)^{n+1} \sum_j (sj+z)^n u^{sj+z-1}.
$$
We divide both sides by $u^{z-1}$ and set $t = u^s$: with
Proposition~\ref{prop:sums-euler-B}, this proves the identity.
\end{proof}

We can also prove a version of the classic identity
$$
\sum_{n=0}^\infty \frac{x^n}{n!}\left(\sum_{k=0}^n A(n,k)t^k\right) = \frac{t-1}{t - e^{x(t-1)}}
$$
and of Ste{\'\i}ngrimsson's similar identity for $E_{s,1}(n,k)$~\cite{steingrimsson}.

\begin{corollary}
$$
\sum_{n=0}^\infty \frac{x^n}{n!}\left(\sum_{k=0}^n \Esz(n,k)t^k\right) =
\frac{(1-t)e^{zx(1-t)}}{1 - te^{sx(1-t)}}.
$$
\end{corollary}

Using Theorem~\ref{thm:F-gf}, we can do the same for $F_s$:
\begin{theorem}
$$
\sum_{n=0}^\infty \frac{x^n}{n!}\left(\sum_{k=0}^n F_s(n,k)u^k\right) =
\frac{u^s - 1}{u - e^{x(u^s-1)}}.
$$
\end{theorem}

\section{Connections with Mastermind}
\label{sec:mastermind}

There are natural connections between Wordle and the classic game Mastermind,
invented by Mordecai Meirowitz.  In Mastermind, 
the setter chooses a ``code'': a string of
length $n$ from an alphabet of size $m$ (which we can write as $[m]$). In each
round, the guesser tries a string in $[m]^n$, and the setter responds with two numbers: the
number of entries that are correct (right number, right place) and the number of additional
entries that are misplaced (right number, wrong place).  This is
a similar setup to Wordle, except that in Wordle we find out specifically which entries
are correct or misplaced, instead of just a count.

Donald
Knuth~\cite{Knuth1976} showed that, for the original game ($n = 4$ and $m = 6$), five rounds
are sufficient for the guesser to guess the code.  Many others have evaluated the
game for various values of $n$ and $m$.
The typical emphasis is on a worst-case analysis; one
can imagine the setter adversarially altering the code to be consistent with all information
presented thus far, with the goal of maximizing the number of rounds.  One reason for
this worst-case analysis is that not all codes are the same;
even if we do something along the lines of
Section~\ref{sec:avg-case}, permuting the colors and the entries, there is a difference between
the code $(1,2,3,4)$ (with no repeated entries) and $(1,1,2,3)$ (with one repeat).
An ``average-case'' analysis only makes sense with respect to some distribution on the
setter's choices.

One Mastermind variant is ``permutation Mastermind'', in which $m = n$ and both the setter's
code and the guesser's guesses must contain each value exactly once.  (In this variant, the
setter responds with the number $c$ of correct entries, since the number misplaced
is simply $n-c$.)  El Ouali et al.~\cite{EOGlSaSr2018} have shown that $n \log_2 n + O(n)$ queries suffice.

Li and Zhu~\cite{LiZhu2024} have analyzed a generalized version of Wordle, with a
code from the alphabet $[m]^n$; essentially, they are playing Mastermind with Wordle-like
feedback. Among other things, they show that, in the permutation setting where
$m = n$ and the code is required to use each value exactly once, then the guesser
needs $n$ guesses in the worst case; this is true even if the guesses are are allowed to
have values multiple times.  This worst-case bound applies directly to our setting; their
result shows that, for any permutation Wordle strategy, there must be some input where
it requires $n$ guesses.

Li and Zhu give a
simple strategy with a matching upper bound of $n$ guesses for the general case $m=n$:
the first $n-1$ guesses are the strings $1^n$, $2^n$, $\ldots$, $(n-1)^n$, at which point
the guesser knows the code.  Our result can be considered an improvement on this simple
strategy, in that it has the same worst-case behavior but better average-case behavior.

\section{Open questions}
\label{sec:conclusion} \label{sec:open}

We have described the game \classic and its generalization \rainbow.  We have also
described the strategy \basicstrat, which we believe solves both games optimally.

The natural open question is whether \basicstrat really is optimal for \classic,
or for \rainbow.  Does Conjecture~\ref{conj:optimal} hold---that is, does \basicstrat
dominate all other strategies, in the sense that, for
all $r$, it wins more often within $r$ rounds than any other strategy?  If not, is it
optimal in some more limited sense (e.g., average number of rounds)?  Suppose we allow the
guesser to guess arbitrary strings of (multicolor) numbers,
as in Li--Zhu~\cite{LiZhu2024}: does that enable a more efficient strategy?

Aurora Hiveley~\cite{Hiveley2025} has done some computer analyses of \basicstrat,
and also proven Conjecture~\ref{conj:optimal} in the case $r \le 3$.  It would be interesting to
see how far her work can be extended.

We introduced the Wordle numbers $W_s(n,r)$, and showed that they simultaneously
refine the Eulerian numbers $A(n,k)$ and the higher-order $\Esz(n,k)$. Do these numbers
count anything on their own, outside of the permutation Wordle setting?

There is a natural connection between the higher-order $\Esz(n,k)$ and the joint
distribution of excedances and fixed points.  More precisely, if we let
$J_n(x,y) = \sum_{\sigma \in S_n} x^{\exc(\pi)} y^{\fixed(\pi)}$, where $\fixed(\pi)$ is the number
of fixed points of $\pi$, then
$$
\sum_t \Esz(n,k) t^k = s^n J_n\left(t, \frac{(s-z)t + z}{s}\right).
$$
Does this joint distribution shed light on these $\Esz$ values?

\section*{Acknowledgments}

The authors are indebted to Alex Miller for making the connection with excedances and for
proposing the two-colored
variant.  We would also like to thank Joe Buhler, Tim Chow, Keith Frankston,
Zachary Hamaker, Ben Howard, Danny Scheinerman, Jeff VanderKam, and Doron Zeilberger for helpful discussions.  And we appreciate the package {\tt wordle.sty}, by
Andrew Mathas and C{\'e}dric Pierquet, which we used to make the Wordle diagrams
throughout the paper.

\appendix

\section{Proving recurrences from generating functions}\label{sec:gf-proofs}

We now prove Theorem~\ref{thm:multi-recur}.

\begin{proof}[Proof of Theorem~\ref{thm:multi-recur}]
When $n=0$, Theorem~\ref{thm:multi-gf} gives
$$
\sum_k W_{()}(0,k) u^k = (1-u)\sum_j u^j = 1,
$$
which matches our initial condition.

For $n > 0$, it is convenient to break up the right-hand side of~\eqref{eq:multi-recur}
into three pieces:
\begin{align*}
X(n,k) & := (k+1)\Dmult{n-1}(n-1,k) - (k+1-s_n)\Dmult{n-1}(n-1,k-s_n), \\
Y(n,k) & := \sum_{r=1}^{s_n} \Dmult{n-1}(n-1,k-r), \\
Z(n,k) & := \sum_{\ell=1}^{n-1} s_\ell W_{\noell}(n-1,k-s_\ell).
\end{align*}
We wish to show that, for all $k$, $\Dmult{n}(n,k) = X(n,k) + Y(n,k) + Z(n,k)$, or
\begin{equation}\label{eq:multi-recur-pf-goal}
\sum_k \Dmult{n}(n,k) u^k = \sum_k X(n,k) u^k + \sum_k Y(n,k) u^k + \sum_k Z(n,k) u^k.
\end{equation}.
We begin with $\sum_k X(n,k) u^k$.  Setting $k' = k-s_n$,
\begin{align*}
\sum_k X(n,k) u^k & = \sum_k (k+1) \Dmult{n-1}(n-1,k) u^k - u^{s_n} \sum_{k'} (k'+1)\Dmult{n-1}(n-1,k')u^{k'}
\\
& = (1 - u^{s_n}) \sum_k (k+1) \Dmult{n-1}(n-1,k) u^k \\
& = (1 - u^{s_n}) \frac{d}{du} \sum_k \Dmult{n-1}(n-1,k) u^{k+1} \\
& = (1 - u^{s_n}) \frac{d}{du} \left[\left(\sum_j (j+1)^{n-1} u^{j+1}\right)
(1-u)\prod_{\ell=1}^{n-1} (1-u^{s_\ell})\right].
\end{align*}
We express this derivative as a sum of three terms:
\begin{multline*}
\sum_k X(n,k) u^k = (1-u^{s_n})\sum_j(j+1)^n u^j (1-u) \prod_{\ell=1}^{n-1}(1-u^{s_\ell}) \\
- (1-u^{s_n})\sum_j(j+1)^{n-1} u^{j+1} \prod_{\ell=1}^{n-1}(1-u^{s_\ell}) \\
- (1-u^{s_n}) \sum_{\ell=1}^{n-1} s_\ell u^{s_\ell-1} \sum_j (j+1)^{n-1} u^{j+1} (1-u) \prod_{m\ne\ell,m=1}^{n-1}(1-u^{s_m}).
\end{multline*}
Each term can now be rewritten using~\eqref{eq:multi-gf}:
\begin{multline*}
\sum_k X(n,k) u^k = \sum_k \Dmult{n}(n,k) u^k - \frac{u(1-u^{s_n})}{1-u} \sum_k \Dmult{n-1}(n-1,k) u^k
\\ - \sum_{\ell=1}^{n-1} s_{\ell} u^{s_\ell} \sum_k W_\noell(n-1,k) u^k.
\end{multline*}
This is
$$
\sum_k X(n,k) u^k = \sum_k \Dmult{n}(n,k) u^k - \sum_k Y(n,k) u^k - \sum_k Z(n,k) u^k,
$$
demonstrating \eqref{eq:multi-recur-pf-goal}, as desired.
\end{proof}

The proof of Theorem~\ref{thm:F-recur} uses similar logic.

\begin{proof}[Proof of Theorem~\ref{thm:F-recur}]
By Theorem~\ref{thm:F-gf},
$$
\sum_k F_s(n,k) u^k = (1-u^s)^{n+1} \sum_j (j+1)^n u^j.
$$
When $n=0$, this gives $\sum F_s(0,k) u^k = (1-u^s)\sum_j u^j = 1 + \cdots + u^{s-1}$, which matches
our initial condition $F_s(0,k) = 1$ for $0 \leq k <s$ and $F_s(0,k) = 0$ otherwise.  For $n > 0$, let
$$
V(n,k) = (k+1)F_s(n-1,k) + (s-k-1)F_s(n-1,k-s).
$$
The calculation
\begin{align*}
\sum_k V(n,k) u^k & = \sum_k (k+1)F_s(n-1,k) u^k - u^s \sum_k (k-s+1) F_s(n-1,k-s) u^{k-s} \\
& = (1-u^s) \sum_k (k+1) F_s(n-1,k) u^k \\
& = (1-u^s) \frac{d}{du} u \sum_k F_s(n-1,k) u^k \\
& = (1-u^s) \frac{d}{du} \left[(1-u^s)^n \sum_j (j+1)^{n-1} u^{j+1}\right] \\
& = (1-u^s)^{n+1} \sum_j (j+1)^n u^j - nsu^{s-1} (1-u^s)^n \sum_j (j+1)^{n-1} u^{j+1} \\
& = \sum_k F_s(n,k) u^k - sn \sum_k F_s(n-1,k) u^{s+k} \\
& = \sum_k (F_s(n,k) - nsF_s(n-1,k-s))u^k,
\end{align*}
completes the proof.
\end{proof}

\bibliographystyle{plain}
\bibliography{shifts}

\end{document}